\documentclass[a4paper,12pt]{amsart}
\usepackage{a4wide}

\usepackage{amssymb,amsmath,latexsym,amsthm}
\usepackage[english]{babel}

\newtheorem{theorem}{Theorem}[section]
\newtheorem{lemma}{Lemma}[section]

\newtheorem{corollary}{Corollary}[section]

\theoremstyle{definition}
\newtheorem*{definition}{Definition}

\numberwithin{equation}{section}

\newcommand{\mes}{\operatorname{mes}}
\newcommand{\bigzero}{\text{\rm\huge0}}

\begin{document}

\title[Density function of real algebraic numbers]{On the density function of the distribution\\of real algebraic numbers}

\author{\sc Denis Koleda}
\address{Denis Vladimirovich Koleda (Dzianis Kaliada)\\
Institute of Mathematics\\
National Academy of Sciences of Belarus\\
220072 Minsk\\
Belarus}
\email{koledad@rambler.ru}



\thanks{Research was partially supported by grant SFB 701}

\maketitle

\begin{abstract}
%Dans cet article, nous \'etudions la distribution des nombres alg\'ebriques r\'eels.
%\'Etant donn\'e un intervalle $I$, un entier positif $n$ et $Q > 1$, on 
%d\'efinit la fonction $\Phi_n(Q; I)$ comme \'etant le nombre de nombres alg\'ebriques dans $I$ de degr\'e $n$ et hauteur na\"ive $\le Q$.
%Soit $I_x = (-\infty,x]$. La fonction de distribution est d\'efinie comme la limite (quand $Q \to \infty$) de $\Phi_n(Q; I_x)$ divis\'e par le nombre total de nombres alg\'ebriques r\'eels de degr\'e $n$ et de hauteur na\"ive $\le Q$.
%Nous montrons que la fonction de distribution existe et est contin\^ument diff\'erentiable.
%Nous donnons aussi une formule explicite pour sa d\'eriv\'ee (d\'enomm\'ee la densit\'e de la distribution).
%Nous \'etablissons une formule asymptotique pour $\Phi_n(Q; I)$ avec des estimations sup\'erieure et inf\'erieure pour le terme d'erreur dans cette formule.
%Il est d\'emontr\'e que ces estimations sont exactes pour~$n \ge 3$.
%Une cons\'equence du th\'eor\`eme principal est le fait que la distribution des nombres r\'eels alg\'ebriques de degr\'e $n \ge 2$ est non uniforme.
%
%\begin{center}
%--- // --- // ---
%\end{center}

In this paper we study the distribution of the real algebraic numbers.
Given an interval $I$, a positive integer $n$ and $Q>1$,
define the counting function $\Phi_n(Q;I)$ to be the number of algebraic numbers in $I$ of degree $n$ and height $\le Q$.
Let $I_x = (-\infty,x]$. The distribution function is defined to be the limit (as $Q\to\infty$) of $\Phi_n(Q;I_x)$ divided by the total number of real algebraic numbers of degree $n$ and height $\le Q$.
We prove that the distribution function exists and is continuously differentiable.
We also give an explicit formula for its derivative (to be referred to as the distribution density) and establish an asymptotic formula for $\Phi_n(Q;I)$ with upper and lower estimates for
the error term in the asymptotic.
These estimates are shown to be exact for~$n \ge 3$.
One consequence of the main theorem is the fact that the distribution of real algebraic numbers of degree $n \ge 2$ is non-uniform.
\end{abstract}

\bigskip

\begin{section}{Introduction and main results}
This paper was inspirited by two famous results: the equidistribution of the Farey fractions and the fact that real algebraic numbers form a regular system. So we briefly describe the background of our investigation.

The classical Farey sequence $\mathcal{F}_Q$ of order $Q$ is formed by irreducible rational fractions in~$[0,1]$ having denominators at most $Q$ and arranged in increasing order:
\[
\mathcal{F}_Q := \left\{\frac{a}{b}: a,b\in\mathbb{Z}, \ 0\le a \le b \le Q, \ \gcd(a,b)=1 \right\}.
\]
The cardinality of $\mathcal{F}_Q$ has the following asymptotics \cite[p. 144, Satz 1]{Wal1963}:
\begin{equation}\label{eq-Walfisz}
\# \mathcal{F}_Q = \frac{3}{\pi^2} Q^2 + O(Q (\ln Q)^{2/3} (\ln\ln Q)^{4/3}).
\end{equation}
The fact that this sequence is uniformly distributed in $[0,1]$ is well-known, that is, for any interval $I\subseteq [0,1]$
\[
\lim_{Q\to\infty}\frac{\# (\mathcal{F}_Q \cap I)}{\# \mathcal{F}_Q} = |I|.
\]
There are several proofs of this fact (see, e.g. \cite{EKKW1940}, \cite{Mik1949}, \cite{Nied1973}). The discrepancy of the Farey sequence is defined as
\[
D_Q := \sup_{\alpha\in[0,1]} \left|\frac{\# (\mathcal{F}_Q \cap [0,\alpha])}{\# \mathcal{F}_Q} - \alpha\right|.
\]
In 1973, H.~Niederreiter \cite{Nied1973} established the true order of the discrepancy: $D_Q \asymp Q^{-1}$.
In~1999, F.~Dress~\cite{Dre1999} found its true value:
\begin{equation}\label{eq-DQ}
D_Q = \frac{1}{Q}.
\end{equation}
For more bibliography, one can see the nice book \cite[Section 2.4]{KuiNied1974}. %by L.~Kuipers and H.~Niederreiter
In 1924, J.~Franel \cite{Fra1924} showed relation between the distribution of Farey series and the Riemann hypothesis. Additional interesting facts and many references can be found in the survey \cite{CobZah2003}. %by C.~Cobeli and A.~Zaharescu

In 1970, A.~Baker and W.~M.~Schmidt \cite{BakSch70} introduced the concept of a regular system and proved that the set of real algebraic numbers of degree at most $n$ forms a \emph{regular system},
that is, there exists a constant $c_n$ depending on $n$ only such that for any interval $I$ for all sufficiently large $Q\in\mathbb{N}$  there exist at least
\[
c_n\,|I|\, Q^{n+1} (\ln Q)^{-3n(n+1)}
\]
algebraic numbers $\alpha_1,\dots,\alpha_k$ of degree at most $n$ and height at most $Q$ satisfying
\[
|\alpha_i-\alpha_j|\geq Q^{-(n+1)} (\ln Q)^{3n(n+1)},\quad 1\leq i<j\leq k.
\]
In 1983, their results were improved by V.~I.~Bernik \cite{Bern1983}.
In 1999, concerning the regularity of the set of real algebraic numbers, V.~V.~Beresnevich \cite{Bere99} showed that the logarithmic factors can be omitted.
In the theory of Diophantine approximation, regular systems arose as a useful tool for calculation the Hausdorff dimension of sets of transcendental numbers allowing approximation by algebraic numbers with a given precision. These problems go back to Mahler's investigations and the Khintchine theorem.
For a more detailed discussion of the literature we refer to the monograph by Y.~Bugeaud~\cite{Bug2004}.

In 1971, H.~Brown and K.~Mahler \cite{BroMah1971} proposed a natural generalization of the Farey sequences for algebraic numbers of higher degrees.
However, till recently, an important question about the distribution of these generalized Farey sequences remained unanswered.
In 1985, in a letter to V.~G.~Sprind\v{z}uk, K.~Mahler noticed that it is unknown, even for the second degree, what is the distribution of algebraic numbers.
In a private talk, V.~I.~Bernik conjectured that the real algebraic numbers are distributed uniformly, he also brought some natural heuristic arguments in favour of his conjecture.

This all motivated the author to start his own investigation that finally resulted in this paper.
Our main result (Theorem \ref{thm1} below) is an asymptotic formula for counting real algebraic numbers of arbitrary fixed degree in any interval.
A short scheme \cite{Ka2012} of a proof for all degrees was published in 2012.
For algebraic numbers of the second degree this question was solved in \cite{Ka2013-3}. 
Here, we present a full proof for arbitrary degrees based on the scheme \cite{Ka2012}.

Note that the generalized Farey sequence \cite{BroMah1971} is based on so called ``naive'' height, which is used in our paper too. This height is not the only height function used in number theory. In many applications, the multiplicative Weil height is extensively used.
There is a number of papers concerning the following problem: over a fixed number field, one needs to count all elements of degree $n$ having multiplicative height at most $T$ as $T$ tends to infinity. For example, results of such type are obtained by D.~Masser and J.~D.~Vaaler in~\cite{MasVaa2007}, \cite{MasVaa2008}.

The results of this paper are closely related to the setting of random polynomials. Most of results in this area concern problems, where the degrees of the polynomials grow to infinity. In the context of the paper, two types of results are of interest.
Going in the first direction (see, e.g., \cite{IbrMas71}, \cite{Kac43}, \cite{Kac48}), the average number of real roots of random polynomials is estimated for different conditions on the distribution of polynomial coefficients. The second direction is to study the distribution of zeros of random polynomials on the complex plane, and the landmark result by P.~Erd\H{o}s and P.~Tur\'{a}n \cite{ErdTur1950} states that the arguments of complex roots of random polynomials are uniformly distributed as the degree tends to infinity. For some general conditions on polynomial coefficients, these roots are clustered near the unit circle \cite{IbrZap2013}.

Turning back to the subject of the article, we say that the V.~I.~Bernik conjecture about uniformity is now disproved: it turned out that the real algebraic numbers of higher degrees are nonuniformly distributed in contrast to the rational numbers. However, the fact of equidistribution of the rational numbers can be obtained as a particular case of the degree one using the way proposed here (see also Remark 1 in Section \ref{sec-rem}).

Now, in order to formulate the main theorem, we introduce some notation and terminology.

Let $p(x) = a_n x^n + \ldots + a_1 x + a_0$ be a polynomial of degree $n$,
and let $H(p)$ be its height defined as $H(p) = \max_{0\le i\le n} |a_i|$.
Let $\alpha\in\mathbb{C}$ be an algebraic number with its corresponding minimal polynomial $p\in\mathbb{Z}[x]$, i.e.,
a polynomial with integer coefficients such that $p(\alpha)=0$, and its degree $\deg(p)$ is minimal, and the greatest common divisor of its coefficients equals to 1.
For an algebraic number $\alpha$ its degree $\deg(\alpha)$ and its height $H(\alpha)$ are defined as the degree and the height of the corresponding minimal polynomial.
Let $\mathbb{A}_n(Q)$ denote the set of real algebraic numbers $\alpha$ of degree $\deg(\alpha)=n$ and height $H(\alpha)\le Q$.

Everywhere $\# M$ denotes the number of elements in a set $M$, and $\mes_{k}M$ denotes 
the $k$--dimensional Lebesgue measure of a set
$M \subset \mathbb{R}^{d}$ ($k \le d$).
The length of an interval $I$ will be denoted as $|I|$.
To denote asymptotic relations between functions, Vinogradov's symbol $\ll$ is going to be used:
the expression $f \ll g$ implies $f \le c_1 g$, where $c_1$ is a positive constant depending only on the degree $n$ of the studied algebraic numbers.
The notation $f \asymp g$ is used for asymptotically equivalent functions, i.e. $g \ll f \ll g$.
The relation $f \ll_{x_1,x_2,\ldots}\, g$ means that implicit constants depend only on quantities $x_1,x_2,\ldots$, and asymptotic equivalence $f \asymp_{x_1,x_2,\ldots}\, g$ is defined similarly.
To emphasize the actual asymptotics of a magnitude $X$, we say that the \emph{true order} of $X$ is $Y$ if $c_1 Y \le |X| \le c_2 Y$ for positive constants $c_1$ and $c_2$ depending, perhaps, on some fixed parameters.

Throughout the paper, we consider polynomials with real coefficients as vectors in Euclidean space.
So the usual Lebesgue measure becomes applicable to sets of polynomials with a fixed degree.

In the case $n=2$, an extra factor $\log Q$ appears in formulas. Therefore, for conciseness, we use the following notation
\[
\ell(n) := \begin{cases}
1, & n = 2,\\
0, & n\ge 3.
\end{cases}
\]

Let $n \in \mathbb{N}$, and $Q>1$. In order to describe the distribution of algebraic numbers, we introduce the following quantity.
\begin{definition}
Let $S\subseteq \mathbb{R}$. For the real algebraic numbers of degree $n$ and height at most $Q$, the \emph{counting function} $\Phi_n(Q,S)$ is defined as
\begin{equation}
\Phi_n(Q, S) = \# (\mathbb{A}_n(Q) \cap S).
\end{equation}
\end{definition}

The asymptotic behavior of the function $\Phi_n(Q,I)$ is described by our main result (as $Q\to\infty$):
\begin{theorem}\label{thm1}
There exists a continuous positive function $\phi_n(x)$ such that for all intervals $I \subseteq \mathbb{R}$
\begin{equation}\label{eq_thm1}
\Phi_n(Q, I) = \frac{Q^{n+1}}{2\zeta(n+1)} \int\limits_{I} \phi_n(x)\, dx + O\!\left(Q^n (\ln Q)^{\ell(n)}\right),
\end{equation}
where $\zeta(x)$ is the Riemann zeta function, and the implicit constant in the big-O notation depends only on the degree $n$.

For infinitely many intervals $I=I(Q)$ the true order of the remainder term is $Q^n$ as $Q\to\infty$.

The function $\phi_n(x)$ has the following properties:
\begin{align}
\phi_n(-x) &= \phi_n(x),  &\forall x \in \mathbb{R}, \label{eq:phieven} \\ 
x^2 \phi_n(x) &= \phi_n\!\left(\frac{1}{x}\right),  &\forall x\in \mathbb{R} \backslash \{0\}, \label{eq:phiinv}
\end{align}
and it can be written explicitly as
\begin{equation}\label{eq5}
\phi_n(x) = \int\limits_{\Delta_n(x)} \left|\sum_{k=1}^n k p_k x^{k-1}\right|\,dp_1\ldots\,dp_n, \qquad x \in \mathbb{R},
\end{equation}
where integration is performed over the region
\[
\Delta_n(x)=\left\{(p_1,\ldots,p_n)\in\mathbb{R}^n : \max\limits_{1\le k\le n} |p_k|\le 1, \
\left|\sum_{k=1}^n p_k x^k \right|\le 1 \right\}.
\]
\end{theorem}

Note that in \eqref{eq_thm1} the endpoints of $I$ are quite arbitrary; for example, they even can be functions of $Q$ or take values $-\infty$ and $+\infty$.

For $n=2$ the error term in \eqref{eq_thm1} is actually $O(\left.\arctan t\right|_{t=a}^b Q^2\ln Q + Q^2)$, where the implicit big-O-constant is absolute (see \cite{Ko2015} for a proof).
But here we use $O(Q^2 \ln Q)$ to omit particular details and give a common proof for all degrees $n\ge 2$.

It is interesting to observe that $2^{-n-1}\phi_n$ is identical with the density function of the real roots of a random polynomial with independent coefficients uniformly distributed in $[-1,1]$ (see~\cite{Za2005}).

\begin{definition}
The function $\phi_n(x)$ will be called the \emph{counting density}.
\end{definition}

Theorem \ref{thm1} can be easily interpreted in terms of probability theory.
\begin{definition}
Let $I_x = (-\infty,x]$. For the real algebraic numbers of degree $n$, the \emph{distribution function} $F_n(x)$ is defined by
\begin{equation*}
F_n(x) = \lim_{Q \to \infty} \frac{\Phi_n(Q, I_x)}{\Phi_n(Q, \mathbb{R})}.
\end{equation*}
\end{definition}
Theorem \ref{thm1} ensures that this limit exists and is differentiable w.r.t. $x$.
\begin{definition}
The function
\begin{equation}\label{eq-distr-dens}
\rho_n(x) = F_n'(x)
\end{equation}
will be called the \emph{distribution density}.
\end{definition}

Clearly, the distribution density differs from the counting density only by a constant factor. The latter one is introduced only to simplify formulas.

%%%%%%%%%%%%%%%%%%%%%%%%%%%%%%%%%%%%%%%%%%%%%%%%%%%%%%%%%%%%%%%%%%%%%%%%%%%%%%%%
\end{section}

\section{Auxiliary lemmas}

\begin{lemma}\label{lm-func-eq}
For a set $S\subset\mathbb{R}$ denote
\[
-S := \{-x: x\in S\}, \qquad S^{-1} := \{x^{-1}: x\in S\}.
\]
For any $S\subset\mathbb{R}$, the counting function $\Phi_n(Q,S)$ satisfies the following equations ($0\not\in S$ in the second one):
\begin{align}
\Phi_n(Q,-S) = \Phi_n(Q,S),  \label{eq:Phi-x}\\
\Phi_n(Q,S^{-1}) = \Phi_n(Q,S). \label{eq:Phixinv}
\end{align}
\end{lemma}

\begin{proof}
To prove the identity \eqref{eq:Phi-x} it is enough to observe that $\alpha\in\mathbb{A}_n(Q)$ if and only if $-\alpha\in\mathbb{A}_n(Q)$ since $P(x) \in \mathbb{Z}[x]$ is equivalent to
$P(-x) \in \mathbb{Z}[x]$, where $\deg P(x) = \deg P(-x)$ and $H(P(x)) = H(P(-x))$.

The property \eqref{eq:Phixinv} follows from a similar argument: $\alpha\in\mathbb{A}_n(Q)$ if and only if
$1/\alpha\in\mathbb{A}_n(Q)$ since $P(x) \in \mathbb{Z}[x]$ is equivalent to
$x^n P\left(\frac{1}{x}\right) \in \mathbb{Z}[x]$, where $n = \deg P(x)$ with $H\left(x^n P\left(\frac{1}{x}\right)\right) = H(P(x))$.
Since $P(x)$ is an irreducible polynomial with $\deg(P)\ge 2$, we have $P(0)\ne 0$ and thus
$\deg x^n P\left(\frac{1}{x}\right) = \deg P(x)$.
\end{proof}

Remark. In terms of $F_n(x)$ equations \eqref{eq:Phi-x} and \eqref{eq:Phixinv} take the form:
\begin{align}
F_n(-x) + F_n(x) &= 1,  \label{eq:F-x}\\
F_n\!\left(\frac{1}{x}\right) + F_n(x) &= \frac{2+\operatorname{sgn}(x)}{2}. \label{eq:Fxinv}
\end{align}

Due to the definition of $F_n(x)$, assuming $I_x = (-\infty,x]$, we immediately obtain \eqref{eq:F-x} from the equality:
\[
-I_x = [-x,+\infty) = \left(\mathbb{R}\setminus (-\infty,-x]\right) \cup \{-x\}.
\]

To prove \eqref{eq:Fxinv}, let us consider the positive and negative values of $x$ separately.

For $x>0$
\begin{equation*}
I_x^{-1} = (-\infty,0)\cup [x^{-1},+\infty) = (-\infty,0) \cup \left(\mathbb{R}\setminus(-\infty,x^{-1}]\right) \cup\{x^{-1}\}.
\end{equation*}

Now let $x<0$. Then
\[
I_x^{-1} = [x^{-1},0) = \left((-\infty,0)\setminus(-\infty,x^{-1}]\right) \cup\{x^{-1}\}.
\]

Note that $F_n(x)$ is continuous and $\lim_{x\to 0^-} F_n(x) = F_n(0)$.

\begin{lemma}[\cite{Dub14}, and \cite{Waer36}]\label{lm_pr_mn}
Let $\mathcal{R}_n(Q)$ be the set of polynomials $p$, where $\deg p = n$, $H(p)\le Q$, and each $p$ is reducible over $\mathbb{Q}$.
For $Q\to \infty$, the cardinality of $\mathcal{R}_n(Q)$ can be asymptotically estimated as
\begin{equation}\label{eq-num-of-reduc}
\#\mathcal{R}_n(Q) \asymp_n Q^n (\ln Q)^{\ell(n)}.
\end{equation}
\end{lemma}

\begin{proof}
For the reader's convenience, we give a concise proof here.

Let $p(x) = f(x)\, g(x)$, where $f$ and $g$ are integral polynomials with $\deg f = n_1$,
$\deg g = n_2$, $H(f) = H_1$, $H(g) = H_2$, $n_1 + n_2 = n$, $n_1\le n_2$.

It is known (see e.g. \cite[Theorem 4.2.2, p. 144]{Pra01}) that
\begin{equation*}
\left(2^{n_1+n_2-2} \sqrt{n_1+n_2+1}\right)^{-1} H(f) H(g) \le H(p) \le (1+n_1) H(f) H(g).
\end{equation*}

Hence, the number of reducible polynomials in 
$\mathcal{R}_n(Q)$ must not exceed the number of pairs $(f,g)$ with the heights bounded by the condition $H_1 H_2 \le c_1(n) Q$.

For fixed $n_1$ and $n_2$, denote by $\mathcal{PP}_{n_1,n_2}(Q)$ the set of pairs of integral polynomials $(p_1, p_2)$ such that
\[
\deg p_i = n_i, \quad H(p_1) H(p_2) \le c_1(n)Q.
\]

Now the proof can be concluded by writing
\[
\#\mathcal{R}_n(Q) \le \#\bigcup_{\substack{1\le n_1\le n_2\\n_1+n_2=n}} \mathcal{PP}_{n_1,n_2}(Q)
\ll_n 
\sum_{\substack{1\le n_1\le n_2\\n_1+n_2=n}}
\left(
\sum_{\substack{1\le H_1,\, H_2\\H_1 H_2 \le c_1(n)Q}} H_1^{n_1} H_2^{n_2}
\right).
\]
The lower bounds can be obtained in a similar way.
\end{proof}

\begin{lemma}[\cite{Dav51_LP}]\label{thm_intp_num}\label{cr1}
For a finite system of inequalities
\[
F_i(x_1,\dots,x_d)\ge 0, \qquad 1\le i\le k,
\]
where each $F_i$ is a polynomial with real coefficients of degree $\deg F_i \le m$,
let $\mathcal{D}\subset \mathbb{R}^d$ be a bounded set of its solutions.
Let
\[
\Lambda(\mathcal{D}) = \mathcal{D}\cap \mathbb{Z}^d.
\]
Then
\[
\left|\#\Lambda(\mathcal{D}) - \mes_d \mathcal{D}\right| \le C \max(\bar{V}(\mathcal{D}), 1),
\]
where the constant $C$ depends only on $d$, $k$, $m$, and $\bar{V}(\mathcal{D})$ is the maximal $r$-dimensio\-nal measure of projections of $\mathcal{D}$ obtained by equating $d-r$ coordinates of the points in $\mathcal{D}$ to zero, where $r$ takes all values from $1$ to $d-1$, i.e.,
\[
\bar{V}(\mathcal{D}) := \max\limits_{1\le r < d}\left\{ \bar{V}_r(\mathcal{D}) \right\}, \quad
\bar{V}_r(\mathcal{D}) := \max\limits_{\substack{\mathcal{J}\subset\{1,\dots,d\} \\ \#\mathcal{J} = r}}\left\{ \mes_r \operatorname{Proj}_{\mathcal{J}} \mathcal{D} \right\},
\]
where $\operatorname{Proj}_{\mathcal{J}} \mathcal{D}$ is an orthogonal projection of $\mathcal{D}$ onto a coordinate subspace formed by coordinates with indices in $\mathcal{J}$.
\end{lemma}

For a bounded set $\mathcal{D} \subset \mathbb{R}^d$, let $Q\cdot\mathcal{D}$ denote the set $\mathcal{D}$ scaled by a factor $Q$:
\[
Q\cdot\mathcal{D} := \{Q{\bf x}\in\mathbb{R}^d : {\bf x}\in\mathcal{D}\}.
\]

The following lemma estimates the number of primitive vectors
${\bf x}=(x_1,\ldots,x_d)\in \mathbb{Z}^d$, i.e., vectors such that $\gcd(x_1,\ldots,x_d)=1$, lying within bounded regions in $\mathbb{R}^d$.

\begin{lemma}\label{lm_skrt}
For a finite set of algebraic inequalities
\[
F_i(x_1,\dots,x_d)\ge 0, \qquad 1\le i\le k,
\]
where each $F_i$ is a polynomial with real coefficients of degree $\deg F_i \le m$,
let $\mathcal{D}\subset [-1,1]^d$ be a bounded set of its solutions.

Let $\lambda(\mathcal{D}, Q)$ be the number of primitive vectors of the lattice $\mathbb{Z}^d$ in the region $Q\cdot\mathcal{D}$.

Then asymptotically we have
\begin{equation}\label{eq-numb-of-p}
\lambda(\mathcal{D}, Q) = Q^d \cdot \frac{\mes_d \mathcal{D}}{\zeta(d)} + O\!\left(Q^{d-1} (\ln Q)^{\ell(d)}\right),
\end{equation}
where $\zeta(x)$ is the Riemann zeta function.
The implicit constant in the big-O notation depends only on the dimension $d$, the size of the system $k$ and the maximal degree $m$ of the algebraic inequalities.
\end{lemma}

Remark. Results of such type are well-known and can be found e.g. in the classical monograph by P. Bachmann \cite[pp. 436--444]{Bach1894} (see formulas (83a) and (83b) on pages~441--442).
For the reader's convenience, we give a short proof here.

\begin{proof}
Without loss of generality, we can exclude the point ${\bf 0}=(0,\ldots,0)$ from counting.

For a positive integer $\nu$, let us count the number of integral points ${\bf x}=(x_1,\ldots,x_d)\in\mathcal{D}$ such that $\nu$ divides $\gcd(x_1,\ldots,x_d)$.
All this points are contained in the lattice $\nu \cdot\mathbb{Z}^d$, and their number in the region $Q\cdot\mathcal{D}$ equals $\#\Lambda\left(\frac{Q}{\nu}\cdot \mathcal{D}\right)$. Applying Lemma~\ref{thm_intp_num}, we have
\begin{equation}\label{eq_LmdQk}
\#\Lambda\left(\frac{Q}{\nu}\cdot\mathcal{D}\right)=\frac{Q^d}{\nu^d} \cdot \mes_d \mathcal{D} +
O\left(\frac{Q^{d-1}}{\nu^{d-1}}\right).
\end{equation}

By applying the inclusion--exclusion principle, we obtain
\begin{multline}\label{eq_lmd_mu}
\lambda(\mathcal{D}, Q)=\#\Lambda(Q\cdot\mathcal{D})-\sum_{q_1} \#\Lambda\left(\frac{Q}{q_1}\cdot\mathcal{D}\right)+
\sum_{q_1<q_2} \#\Lambda\left(\frac{Q}{q_1 q_2}\cdot\mathcal{D}\right)-\ldots+\\
+(-1)^n \sum_{q_1<q_2<\ldots<q_n} \#\Lambda\left(\frac{Q}{q_1 q_2 \ldots q_n}\cdot\mathcal{D}\right)+\ldots =
\sum_{n=1}^\infty \mu(n)\, \#\Lambda\left(\frac{Q}{n}\cdot\mathcal{D}\right),
\end{multline}
where the sums are taken over prime numbers $q_1,q_2,\ldots,q_n$;
$\mu(x)$ is the M\"obius function.

Clearly, for $n > Q$, the lattice $n\cdot\mathbb{Z}^d$ doesn't contain any non-zero points lying in $Q\cdot\mathcal{D}$,
i.e. $\#\Lambda\left(\frac{Q}{n}\cdot\mathcal{D}\right)=0$. Therefore, from \eqref{eq_LmdQk} and 
\eqref{eq_lmd_mu}, we have
\begin{equation*}
\lambda(\mathcal{D}, Q) = Q^d \cdot \mes_d \mathcal{D} \sum_{n=1}^{Q} \frac{\mu(n)}{n^d}+
O\left(Q^{d-1}\sum_{n=1}^{Q}\frac{1}{n^{d-1}}\right).
\end{equation*}

Now applying the well-known facts that for $d\ge 2$
\[
\left|\sum_{n=Q+1}^{\infty} \frac{\mu(n)}{n^d}\right| \, \le \sum_{n=Q+1}^{\infty} \frac{1}{n^d}
\asymp \frac{1}{Q^{d-1}}
\]
as $Q\to\infty$, and that $\sum_{n=1}^\infty \mu(n) n^{-d}=(\zeta(d))^{-1}$, where
$\zeta(d)$ is the Riemann zeta function, completes the proof.
\end{proof}

\begin{lemma}\label{lm-jacob}
Consider a polynomial relation
\begin{equation}\label{eq-p-repr}
a_n x^n +\ldots+a_1 x + a_0 = (x-\alpha)(x-\beta)(b_{n-2}x^{n-2}+\ldots+b_1 x + b_0).
\end{equation}
Then for the vectors $(a_n, \dots, a_1, a_0)$ and $(b_{n-2},\dots, b_1, b_0, \alpha, \beta)$, we can write
\begin{equation}\label{eq_zamena2}
\left(\begin{array}{l}
a_n\\
a_{n-1}\\
a_{n-2}\\
\vdots\\
a_1\\
a_0
\end{array}\right) =
\left(\begin{array}{cccc}
1 & & \phantom{\ddots} & \bigzero\\
-(\alpha+\beta) & 1 & \phantom{\ddots} &\\
\alpha\beta & -(\alpha+\beta) & \ddots &\\
 & \alpha\beta & \phantom{\ddots} & 1\\
 & & \ddots & -(\alpha+\beta) \\
\bigzero & & \phantom{\ddots} & \alpha\beta 
\end{array}\right) \cdot
\left(\begin{array}{l}
b_{n-2}\\
b_{n-3}\\
\vdots\\
b_1\\
b_0
\end{array}\right),
\end{equation}
and the Jacobian $|\det J| = \left| \frac{\partial (a_n,\ldots,a_2,a_1,a_0)}{\partial (b_{n-2},\ldots,b_0;\alpha,\beta)} \right|$ satisfies:
\begin{equation}
|\det J| = |\alpha-\beta| \cdot |g(\mathbf{b},\alpha) g(\mathbf{b},\beta)|,
\end{equation}
where $g(\mathbf{b},x) = b_{n-2}x^{n-2}+\ldots+b_1 x + b_0$.
\end{lemma}
\begin{proof}
By the definition, we have
\[
\det J =
\left|
\begin{array}{cccccc}
\frac{\partial a_n}{\partial b_{n-2}} & \frac{\partial a_n}{\partial b_{n-3}}  & \dots & \frac{\partial a_n}{\partial b_0} & \frac{\partial a_n}{\partial \alpha} & \frac{\partial a_n}{\partial \beta} \\
\frac{\partial a_{n-1}}{\partial b_{n-2}} & \frac{\partial a_{n-1}}{\partial b_{n-3}} & \dots & \frac{\partial a_{n-1}}{\partial b_0} & \frac{\partial a_{n-1}}{\partial \alpha} & \frac{\partial a_{n-1}}{\partial \beta} \\
\vdots & \vdots & & \vdots & \vdots & \vdots \\
\frac{\partial a_0}{\partial b_{n-2}} & \frac{\partial a_0}{\partial b_{n-3}} & \dots & \frac{\partial a_0}{\partial b_0} & \frac{\partial a_0}{\partial \alpha} & \frac{\partial a_0}{\partial \beta} \\
\end{array}
\right|.
\]
In the first row, all the entries, except $\frac{\partial a_n}{\partial b_{n-2}}=1$, are equal to zero.
Using Laplace's formula along this row reduces the dimension of this determinant by~1.
Subtracting the $(n-2)$-th column from the $(n-1)$-th and dividing the result by $(\alpha-\beta)$ yields
\begin{equation}\label{eq22}
\det J =
(\alpha-\beta) \cdot \left|
\begin{array}{cccccc}
1 & & \phantom{\ddots} & & -b_{n-2} & 0\\
-(\alpha+\beta) & 1 & \phantom{\ddots} & & -b_{n-3} & b_{n-2}\\
\alpha\beta & -(\alpha+\beta) & \ddots & & \vdots & b_{n-3}\\
 & \alpha\beta & \ddots & 1 & -b_1 & \vdots\\
 & & \ddots & -(\alpha+\beta) & -b_0 & b_1\\
 & & \phantom{\ddots} & \alpha\beta & 0 & b_0
\end{array}
\right|.
\end{equation}
It is easy to see that the determinant in the left-hand side of \eqref{eq22} is equal to the resultant $R(f,g)$ of the polynomials $f(x)=(x-\alpha)(x-\beta)$ and $g(x)= b_{n-2}x^{n-2}+\ldots+b_1 x + b_0$, up to a sign.
This proves the lemma since $R(f,g) = g(\alpha)g(\beta)$ (see, e.g. \cite[\S 35]{Waer}).
\end{proof}

\begin{lemma}\label{lm_Fbound}
Let $I=[a,b)$ be an interval of length at most $1$.
Let $\mathcal{M}_n(I)$ be the set of polynomials $p\in\mathbb{R}[x]$ with $\deg p = n$ and $H(p)\le 1$ that have at least two roots in~$I$.
Then
\begin{equation*}
\mes_{n+1} \mathcal{M}_n(I) \le \frac{c_2(n)}{\rho^6} |I|^3,
\end{equation*}
where $\rho = \rho(I) = \max(1,|a+b|/2)$, and $c_2(n)$ is a constant that depends only on $n$.
\end{lemma}

\begin{proof}
Let us find an upper bound for
\begin{equation*}
\mes_{n+1} \mathcal{M}_n(I) = \int\limits_{\mathbf{a}\in \mathcal{M}_n(I)} d{\bf a}.
\end{equation*}
Any polynomial in $\mathcal{M}_n(I)$ can be written in the form \eqref{eq-p-repr} with $\alpha, \beta \in I$.
So we use the substitution \eqref{eq_zamena2} to evaluate this integral. The condition ${\bf a} \in \mathcal{M}_n(I)$ is equivalent to the system of inequalities
\begin{equation}\label{eq_grM2}
\left\{
\begin{array}{l}
|a_n| = |b_{n-2}| \le 1,\\
|a_{n-1}| = |b_{n-3} - (\alpha+\beta) b_{n-2}| \le 1,\\
|a_k| = |b_{k-2} - (\alpha+\beta) b_{k-1} + \alpha\beta b_k| \le 1, \ \ \ k = 2,\ldots,n-2,\\
|a_1| = |- (\alpha+\beta) b_0 + \alpha\beta b_1| \le 1,\\
|a_0| = |\alpha\beta b_0| \le 1,\\
a \le \alpha < b, \\
a \le \beta < b.
\end{array}
\right.
\end{equation}

Using Lemma \ref{lm-jacob}, we obtain
\begin{equation}\label{eq-M-ineq}
\mes_{n+1} \mathcal{M}_n(I) \le \int\limits_{\mathcal{M}_n^*(I)} |\alpha-\beta| \cdot |g(\mathbf{b}, \alpha) g(\mathbf{b},\beta)|\, d\mathbf{b}\, d\alpha\, d\beta,
\end{equation}
where $\mathcal{M}_n^*(I)$ is new domain of integration defined by \eqref{eq_grM2} and $g(\mathbf{b},x) = b_{n-2}x^{n-2}+\ldots+b_1 x + b_0$.

This expression cannot be written as an equality since polynomials can have three or more roots in the interval $I$, and then several representations of the form \eqref{eq-p-repr} will exist. If a polynomial has $k>2$ different roots on $I$, then there exist $\binom{k}{2}$ different representations of this type.

Let $I\subset [-2,2]$. Then for all $\alpha, \beta \in I$ we can write
\[
|\alpha+\beta| \le 4, \qquad |\alpha\beta| \le 4,
\]
and therefore for any $(b_{n-2},\dots, b_1, b_0, \alpha, \beta) \in \mathcal{M}_n^*(I)$ we have
\begin{equation}\label{eq-M*bound}
\left\{
\begin{array}{l}
|b_{n-2}| \le 1,\\
|b_{n-3}| \le 1 + 4 |b_{n-2}|,\\
|b_{k-2}| \le 1 + 4 |b_{k-1}| + 4 |b_k|, \qquad k = 2,\ldots,n-2,\\
|\alpha| \le 2, \\
|\beta| \le 2.
\end{array}
\right.
\end{equation}
In other words, for $I\subset[-2,2]$ the domain $\mathcal{M}_n^*(I)$ is enclosed within some box, whose dimensions are determined by $n$ only.

Let us rewrite the multiple integral in \eqref{eq-M-ineq} as follows:
\[
\mes_{n+1} \mathcal{M}_n(I) \le \int\limits_{I\times I} |\alpha-\beta|\, d\alpha\, d\beta \int\limits_{\mathcal{M}_n^*(\alpha,\beta)} |g(\mathbf{b},\alpha) g(\mathbf{b},\beta)|\,d\mathbf{b},
\]
where $\mathcal{M}_n^*(\alpha,\beta) = \{{\bf b} \in \mathbb{R}^{n-1} : \|A(\alpha,\beta) \mathbf{b}\|_\infty \le 1 \}$,
and $A(\alpha, \beta)$ is the $(n+1)\times(n-1)$ matrix from \eqref{eq_zamena2}.

Let us denote $G({\bf b},\alpha,\beta)=g(\mathbf{b},\alpha) g(\mathbf{b},\beta)$, and
\[
\psi(\alpha,\beta)=\int\limits_{\mathcal{M}_n^*(\alpha,\beta)}|G({\bf b},\alpha,\beta)| d{\bf b}.
\]

As stated above, for $\alpha,\beta \in [-2,2]$ the domain $\mathcal{M}_n^*(\alpha,\beta)$ lies within a box of dimensions determined only by $n$. The function $G(\mathbf{b}, \alpha, \beta)$ is a polynomial, and its values cannot exceed a certain constant determined only by $n$ for all $\mathbf{b}\in \mathcal{M}_n^*(\alpha,\beta)$ and all $\alpha, \beta\in I$.
Hence, there exists a constant $c_2(n)$ which depends only on $n$ such that $0 < \psi(\alpha,\beta) \le c_2(n)$ for all $\alpha,\beta \in I$. Thus, we obtain for $I\subset [-2,2]$
\[
\mes_{n+1} \mathcal{M}_n(I) \le c_2(n) |I|^3.
\]

Let $I\subset \mathbb{R}\setminus [-1,1]$.
In this case, we substitute the polynomial $p(x) = a_n x^n + \ldots + a_1 x + a_0$ by the polynomial $q(x) = x^n p(1/x) = a_0 x^n + \ldots + a_{n-1} x + a_n$, and the interval $I$ by $I^* = (1/b, 1/a) \subset [-1,1]$, $ab>0$.
Clearly, under these substitutions the number of roots remains invariant.

Now we can apply the substitution \eqref{eq_zamena2} to the vector $(a_0, a_1, \dots, a_n)$, which leads to
\[
\mes_{n+1} \mathcal{M}_n(I) = \mes_{n+1} \mathcal{M}_n(I^*) \le c_2(n) |I^*|^3 = \frac{c_2(n)}{|ab|^3} |I|^3,
\]
proving the lemma.
\end{proof}

\begin{lemma}\label{lm_emptyinterv}
Let $x_0=a/b$ with $a\in\mathbb{Z}$, $b\in\mathbb{N}$ and $\gcd(a,b)=1$.
Then there are no algebraic numbers $\alpha$ of degree $\deg \alpha = n$ and height $H(\alpha)\le Q$ in the interval $|x-x_0|\le r_0$, where
\[
r_0=r_0(x_0, Q)=\frac{c_3(n)}{b^n Q},
\]
and $c_3(n)$ is an effective constant depending on $n$ only.

A similar fact can be stated for a neighborhood of infinity: no algebraic number $\alpha$ of degree $\deg(\alpha)=n$ and height $H(\alpha)\le Q$ lies in the set
\[
\{x\in\mathbb{R}: |x|\ge Q+1\}.
\]
\end{lemma}

Note that the statement of the lemma implies that
$\Phi_n(Q,S) = 0$ if $S \cap (-Q-1,Q+1) = \varnothing$, and
$\Phi_n(Q,S) = \#\mathbb{A}_n(Q)$ if $(-Q-1,Q+1)\subseteq S$.

\begin{proof}
Let $p(x)=a_n x^n+\ldots+a_1 x+a_0\in\mathbb{Z}[x]$ with $H(p)\le Q$.

We develop $p(x)$ into the Taylor series:
\begin{equation}\label{eq_px}
p(x)=p(x_0)+\sum_{k=1}^n \frac{p^{(k)}(x_0)}{k!} (x-x_0)^k.
\end{equation}

Let $p(x_0)\ne 0$. Then
\begin{equation}\label{eq_pa}
|p(x_0)| \ge \frac{1}{b^n}.
\end{equation}

Assuming $|x_0|\le 1$, we have $|p^{(k)}(x_0)| \ll_n H(p)$ for $k=0,\ldots,n$,
and thus
\begin{equation}\label{eq_dpa}
\left|\sum_{k=1}^n \frac{p^{(k)}(x_0)}{k!} (x-x_0)^k\right| \ll_n
H(p)\,|x-x_0|\sum_{k=0}^{n-1}|x-x_0|^k.
\end{equation}

Development \eqref{eq_px} and estimates \eqref{eq_pa}, \eqref{eq_dpa} imply that any integral polynomial of degree~$n$ and height at most $H$ has no roots $x\ne x_0$ in the circle
\begin{equation}\label{eq_xa}
|x-x_0| < \frac{c_3(n)}{b^n H},
\end{equation}
where $c_3(n)$ is an effective constant.

The case $|x_0|>1$ can be reduced to the case $|x_0| < 1$ with the use of mapping $x\to x^{-1}$.

Let $|\beta|\ge Q+1$. Then
\[
|p(\beta)| \ge |\beta|^n - Q|\beta|^{n-1} - \ldots - Q|\beta|-Q \ge 1.
\]
Thus, the number $\beta$ cannot be a root of the polynomial $p(x)$.
The lemma is proved.
\end{proof}

%%%%%%%%%%%%%%%%%%%%%%%%%%%%%%%%%%%%%%%%%%%%%%%%%%%%%%%%%%%%%%%%%%%%%%%%%%%%%%%%

\section{The proof of the main theorem}

Recall that $I=(\alpha,\beta]$ is a bounded interval.
Let \emph{prime} polynomials be defined as irreducible primitive polynomials with positive leading coefficients.
Clearly, the distribution of algebraic numbers can be expressed in terms of prime polynomials.
Let $N_n(Q,k,I)$ denote the number of prime polynomials $p$ of degree $\deg p = n$ and height $H(p)\le Q$ which have
exactly $k$ roots in the set $I$.
Clearly, we have
\begin{equation}
\Phi_n(Q,I) = \sum_{k=1}^n k N_n(Q, k, I).
\end{equation}

Let $\mathcal{G}_n(k,S)$ denote the set of polynomials $p\in\mathbb{R}[x]$ with $\deg p = n$ and $H(p)\le 1$ that have exactly $k$ roots in a set $S$.
From Lemmas \ref{lm_pr_mn} and \ref{lm_skrt}, we have:
\[
N_n(Q,k,I) = \frac{\mes_{n+1} \mathcal{G}_n(k,I)}{2\zeta(n+1)}\cdot Q^{n+1}  + O\!\left(Q^n (\ln Q)^{\ell(n)}\right),
\]
where the implicit constant in the big-O notation depends only on $n$.

\begin{lemma}
The function $\widehat{\Phi}_n(S)$ defined as
\[
\widehat{\Phi}_n(S) = \sum_{k=1}^n k \mes_{n+1} \mathcal{G}_n(k,S)
\]
is additive and bounded for all $S \subseteq \mathbb{R}$.
\end{lemma}
\begin{proof}
If $S_1 \cap S_2 = \varnothing$, then
\begin{multline*}
\widehat{\Phi}_n(S_1) + \widehat{\Phi}_n(S_2) = \sum_{k=1}^\infty \sum_{m=0}^\infty k \mes_{n+1}\left(\mathcal{G}_n(k,S_1) \cap \mathcal{G}_n(m,S_2)\right) + \\
+ \sum_{m=1}^\infty \sum_{k=0}^\infty m \mes_{n+1}\left(\mathcal{G}_n(k,S_1) \cap \mathcal{G}_n(m,S_2)\right) = \\
= \sum_{\nu=1}^\infty \nu \sum_{k=0}^\nu \mes_{n+1}\left(\mathcal{G}_n(k,S_1) \cap \mathcal{G}_n(\nu-k,S_2)\right) = \widehat{\Phi}_n(S_1 \cup S_2).
\end{multline*}
Here we have used the facts that $\mathcal{G}_n(k,S) = \varnothing$ for $k>n$ and that
\[
\mathcal{G}_n(\nu,S_1\cup S_2) = \bigcup_{k=0}^\nu \left(\mathcal{G}_n(k,S_1) \cap \mathcal{G}_n(\nu-k,S_2)\right).
\]
For any $S\subseteq \mathbb{R}$, we have
\[
\widehat{\Phi}_n(S)\le n\sum_{k=1}^n \mes_{n+1}\mathcal{G}_n(k,S) \le n \mes_{n+1} ([-1,1]^{n+1}) = n 2^{n+1}.
\]
The lemma is proved.
\end{proof}

Let us prove that $\widehat{\Phi}_n(I)$ can be written as
the integral of a continuous positive function over $I$.
Define
\begin{equation}\label{eq-DI-def}
\mathcal{D}(I) = \left\{{\bf p}\in\mathbb{R}^{n+1} : p(\alpha)p(\beta)< 0, \ H(p)\le 1\right\},
\end{equation}
where ${\bf p} = (p_n,\ldots,p_1,p_0)$ is the vector form of a polynomial $p(x)=p_n x^n+\ldots+p_1 x+p_0$.
Clearly, for any vector in $\mathcal{D}(I)$ the corresponding polynomial has the odd number of roots lying in the interval $I$.

From Lemma \ref{lm_Fbound}, we have
\begin{equation}\label{eq-dPhi-D}
\widehat{\Phi}_n(I) = \mes_{n+1} \mathcal{D}(I) + O(|I|^3),
\end{equation}
where the implicit constant in the big-O notation depends only on the degree $n$.

Now let us calculate $\mes_{n+1} \mathcal{D}(I) = \int_{\mathcal{D}(I)} d{\bf p}$.
The domain $\mathcal{D}(I)$ can be defined by the following system of inequalities
\begin{equation}\label{eq-sys}
\begin{cases}
|p_i| \le 1, \qquad 0\le i\le n,\\
f_*(p_1,\dots,p_n) \le p_0 \le f^*(p_1,\dots,p_n),
\end{cases}
\end{equation}
where
\begin{align*}
f_*(p_1,\dots,p_n) &= \min\left\{-\sum_{k=1}^n p_k \alpha^k, -\sum_{k=1}^n p_k \beta^k\right\},\\
f^*(p_1,\dots,p_n) &= \max\left\{-\sum_{k=1}^n p_k \alpha^k, -\sum_{k=1}^n p_k \beta^k\right\}.
\end{align*}

To simplify notation, define a function $h$ as
\begin{multline*}
h(p_1,\ldots,p_n) := f^*(p_1,\dots,p_n) - f_*(p_1,\dots,p_n) = \\
= (\beta-\alpha)\left| \sum_{k=1}^n p_k \sum_{i=0}^{k-1} \alpha^i \beta^{k-i-1} \right|.
\end{multline*}
Consider regions
\[
\Delta_* := \Delta_n(\alpha) \cap \Delta_n(\beta), \qquad
\Delta^* := \Delta_n(\alpha) \cup \Delta_n(\beta),
\]
where
\[
\Delta_n(x)=\left\{(p_1,\ldots,p_n)\in\mathbb{R}^n : \max\limits_{1\le i\le n} |p_i|\le 1, \
\left|\sum_{k=1}^n p_k x^k \right|\le 1 \right\}.
\]
For all $(p_1,\dots,p_n)\in \Delta_*$, both the inequalites
\[
|f_*(p_1,\dots,p_n)| \le 1, \qquad |f^*(p_1,\dots,p_n)| \le 1
\]
hold, and thus, the bound $|p_0|\le 1$ gives no effect in \eqref{eq-sys}. For any $(p_1,\dots,p_n)\not\in \Delta^*$,
the system of inequalities \eqref{eq-sys} is inconsistent for $\alpha$ and $\beta$ being close enough.
Hence, we have the estimate
\begin{equation*}
\int\limits_{\Delta_*} h(p_1,\ldots,p_n)\,dp_1\ldots dp_n
\ \le \ \mes_{n+1} \mathcal{D}(I) \ \le \
\int\limits_{\Delta^*} h(p_1,\ldots,p_n)\,dp_1\ldots dp_n.
\end{equation*}
From above, it follows that
\begin{multline*}
\left|\mes_{n+1} \mathcal{D}(I) \ - \int\limits_{\Delta_n(\alpha)} h(p_1,\ldots,p_n)\,dp_1\ldots dp_n\right| \le \\
\le \int\limits_{\Delta^*\setminus \Delta_*} h(p_1,\ldots,p_n)\,dp_1\ldots dp_n.
\end{multline*}
It is easy to show that the difference of $\Delta^*$ and $\Delta_*$ has a small measure
\[
\mes_n \left(\Delta^* \setminus \Delta_*\right) = O(\beta-\alpha), \qquad \beta\to \alpha,
\]
where the implicit constant in the big-O notation depends on the degree $n$ only.

Therefore, as $\beta$ tends to $\alpha$,
we obtain for any $\alpha\in\mathbb{R}$
\begin{equation}\label{eq-limit-DI}
\lim_{\beta\to\alpha} \frac{\mes_{n+1} \mathcal{D}(I)}{(\beta-\alpha)} \ = \
\int\limits_{\Delta_n(\alpha)} \left|\sum_{k=1}^n k p_k \alpha^{k-1}\right|\,dp_n\ldots dp_1.
\end{equation}

Thus, \eqref{eq-dPhi-D} can be rewritten as
\[
\widehat{\Phi}_n(I) = \phi_n(\alpha) |I| + o(|I|),
\]
where $\phi_n(\cdot)$ is defined by \eqref{eq5}.
Therefore, we have
\[
\widehat{\Phi}_n(I)  = \int\limits_I \phi_n(x) dx.
\]
The statements \eqref{eq_thm1} and \eqref{eq5} are proved.
It is also easy to see that $\phi_n(x)$ is a continuous positive function on $\mathbb{R}$.

The upper bound for the remainder term in \eqref{eq_thm1} is obtained from the error term in \eqref{eq-numb-of-p} and the bound \eqref{eq-num-of-reduc} for the number of reducible polynomials.

Lemma \ref{lm_emptyinterv} gives examples of intervals, for which the remainder term in \eqref{eq_thm1} has the order $O(Q^n)$.
Namely, for every rational $x_0$, there exist an interval with midpoint $x_0$ and length $|I|\asymp_{x_0} Q^{-1}$ that contains no algebraic numbers~$\alpha$ with $\deg(\alpha) = n$ and $H(\alpha)\le Q$.
Therefore, for this interval, the error in \eqref{eq_thm1} is equal to the main term, which is of the order $O(Q^n)$.
The rational numbers are everywhere dense in $\mathbb{R}$.
Hence, we have the second statement of Theorem \ref{thm1}.

Differentiating \eqref{eq:F-x} and \eqref{eq:Fxinv} w.r.t. $x$ yields:
\[
\rho_n(-x) = \rho_n(x), \qquad \frac{1}{x^2}\,\rho_n\!\left(\frac{1}{x}\right) = \rho_n(x).
\]
As stated earlier, $\rho_n(x) = \gamma_n^{-1} \phi_n(x)$,
where $\gamma_n := \int_\mathbb{R} \phi_n(x)\,dx$.
So we have \eqref{eq:phieven} and~\eqref{eq:phiinv}.
The proof is now complete.

%%%%%%%%%%%%%%%%%%%%%%%%%%%%%%%%%%%%%%%%%%%%%%%%%%%%%%%%%%%%%%%%%%%%%%%%%%%%%%%%

\section{Final remarks}\label{sec-rem}

\paragraph{\bf Remark 1}
As noted above, the rational numbers, i.e. the algebraic numbers of the first degree, are distributed uniformly in the interval $[-1,1]$.
Following the way proposed in the paper, one can easily obtain
\begin{equation}\label{eq-Phi1}
\Phi_1(Q,I) = 
\frac{Q^2}{2\zeta(2)} \int\limits_I \phi_1(x)\,dx + O(Q\ln Q),
\end{equation}
where
\[
\phi_1(x) = \frac{1}{\max(1,x^2)}.
\]
This formula corresponds to the general equation \eqref{eq5} when $n=1$ and agrees with the well-known result on the distribution of the Farey sequence (see \cite{Mik1949} for an elementary proof).

In the remainder term in \eqref{eq-Phi1}, the logarithmic factor comes from Lem\-ma \ref{lm_skrt}.
Here we show that this factor actually vanishes as $|I|\to 0$.
A similar effect exists in the case $n=2$ (see \cite{Ko2015} for details).

By definition $\Phi_1(Q,\mathbb{R}) = \#\mathbb{A}_1(Q)$.
Obviously,
\[
\mathbb{A}_1(Q) = \left\{\frac{a}{b} : a\in\mathbb{Z},\ b\in\mathbb{N},\ \gcd(a,b)=1,\ \max(|a|,|b|)\le Q\right\}.
\]
It easy to observe that $\#\mathbb{A}_1(Q) = 4(\#\mathcal{F}_Q - 2) + 3$
where the addition of~$3$ corresponds to the three points: $\pm 1$ and $0$. So we have
\begin{equation}\label{eq-Phi1R}
\Phi_1(Q,\mathbb{R}) = 4 \#\mathcal{F}_Q - 5.
\end{equation}

\begin{theorem}
\begin{equation}\label{eq-ext-Farey}
\sup_{I\subseteq \mathbb{R}}\left|\frac{\Phi_1(Q,I)}{\Phi_1(Q,\mathbb{R})} - \frac14 \int\limits_I \phi_1(x)\,dx\right| \asymp \frac{1}{Q},
\end{equation}
where the supremum is taken over the intervals of the real line.
%in the symbol~$\asymp$ the implicit constant is absolute.
\end{theorem}
\begin{proof}
From \eqref{eq-DQ} and \eqref{eq-Phi1R} we obtain at once
\[
\sup_{I\subseteq [0,1]} \left|\frac{\Phi_1(Q,I)}{\Phi_1(Q,\mathbb{R})} - \frac14 |I|\right| \asymp \frac{1}{Q}.
\]

Let $I\subseteq [1,+\infty)$ and $I^{-1} := \left\{x^{-1} : x\in I\right\}$. Obviously, $\Phi_1(Q,I) = \Phi_1(Q,I^{-1})$ and $I^{-1}\subset(0,1]$, and
$|I^{-1}| = \int_I \frac{dx}{x^2}$.
Therefore,
\[
\sup_{I\subseteq [1,+\infty)} \left|\frac{\Phi_1(Q,I)}{\Phi_1(Q,\mathbb{R})} - \frac14\int_I \frac{dx}{x^2}\right| \le
\sup_{I\subseteq [0,1]} \left|\frac{\Phi_1(Q,I)}{\Phi_1(Q,\mathbb{R})} - \frac14 |I|\right|.
\]

For any other placement of intervals proof can be easily deduced from the cases considered above.
\end{proof}

Using \eqref{eq-Walfisz}, we readily obtain from \eqref{eq-ext-Farey}:
\begin{align}
\Phi_1(Q,I) =& \frac{3}{\pi^2}\, Q^2 \int\limits_I \frac{dx}{\max(1,x^2)} \\
+& O\!\left(Q (\ln Q)^{2/3} (\ln\ln Q)^{4/3} \int\limits_I \frac{dx}{\max(1,x^2)} + Q\right). \nonumber
\end{align}
Note that in the remainder term the integral vanishes as $|I|\to 0$.

\medskip

\paragraph{\bf Remark 2}
The counting density on the interval $|x|\le \frac{\sqrt{2}-1}{\sqrt{2}}\approx 0.29$
can be expressed as
\begin{equation}\label{eq-analit-f}
\phi_n(x) = \frac{2^{n-1}}{3} \left(3+\sum_{k=1}^{n-1} (k+1)^2 x^{2k}\right).
\end{equation}
\begin{proof}
To simplify the calculation of \eqref{eq5}, let us place the following restriction on $x$: $|p_n x^n + \ldots + p_2 x^2 + p_1 x| \le 1$ for all $p_i$ such that $\max_{1\le i\le n} |p_i|\le 1$. This is equivalent to the inequality $|x|\le x_0(n)$, where $x_0(n)$ is the positive solution of the equation $x^n + \ldots + x^2 + x = 1$. For $|x|\le x_0(n)$, we have
\[
\phi_n(x) = 2 \int\limits_{\widetilde{\Delta}_n(x)} \left(\sum_{k=1}^n k p_k x^{k-1}\right) dp_1\ldots\,dp_n,
\]
where
\[
\widetilde{\Delta}_n(x) = \left\{(p_1,\ldots,p_n) \in \mathbb{R}^n : \ \max\limits_{1\le k\le n} |p_k| \le 1, \ \sum_{k=1}^n k p_k x^{k-1} \ge 0 \right\}.
\]
Let $f_n(x; p_n, \dots, p_2)$ denote the function $n p_n x^{n-1} + \ldots + 2 p_2 x$. The variable $p_1$ assumes the values between $\max(-f_n(x; p_n,\dots, p_2), -1)$ and $1$. Assume $f_n(x; p_n,\dots,p_2)\le 1$ for all $p_i$ such that $\max_{2\le i\le n} |p_i|$. This restriction is equivalent to $|x|\le x_1(n)$, where $x_1(n)$ is the positive solution of $n x^{n-1} + \ldots + 2x = 1$. Clearly, $x_1(n)\le x_0(n)$ and $\lim_{n\to\infty} x_1(n) = \frac{\sqrt{2}-1}{\sqrt{2}}$.

For $|x|\le x_1(n)$ we have
\begin{align*}
\phi_n(x) &= 2 \int_{\substack{|p_i|\le 1\\2\le i\le n}} \left(\int_{-f_n(x;p_n,\dots,p_2)}^1 \!\left(f_n(x;p_n,\dots, p_2) + p_1\right) dp_1\right) dp_2\ldots\,dp_n \\
&= \int_{\substack{|p_i|\le 1\\2\le i\le n}} \left(f_n(x;p_n,\dots, p_2)+1\right)^2 dp_2\ldots\,dp_n.
\end{align*}
From the symmetry, the integrals of all the non-square terms are equal to zero, which leads to~\eqref{eq-analit-f}.
\end{proof}
Now it is easy to see that the distribution of algebraic numbers of degree $n$, $n\ge 2$, is non-uniform.
However, rational numbers, i.e., algebraic numbers of the first degree, are distributed uniformly in the interval $[-1,1]$ (see Remark 1).

For $\phi_2(x)$, a piecewise formula involving only rational functions is obtained in \cite{Ko2015}.

\medskip

\paragraph{\bf Remark 3}
It is possible to generalize Theorem \ref{thm1} to height functions different from~$H(p)$. In general, a height function is defined as follows.

\begin{definition}
A function $\mathfrak{h}: \mathbb{R}^{n+1} \to [0,+\infty)$ satisfying the conditions

\begin{enumerate}
\item $\mathfrak{h}(t\mathbf{v}) = |t|\, \mathfrak{h}(\mathbf{v})$ for all $t\in\mathbb{R}$ and all $\mathbf{v}\in\mathbb{R}^{n+1}$;
\item the set $\{\mathbf{v}\in\mathbb{R}^{n+1} : \mathfrak{h}(\mathbf{v}) \le 1\}$ is a convex body;
\item $\mathfrak{h}(\mathbf{v}) = 0$ if and only if $\mathbf{v} = 0$;
\item $\mathfrak{h}\left(v_n, -v_{n-1},\dots, (-1)^{n-1}v_1, (-1)^n v_0\right) = \mathfrak{h}(v_n, v_{n-1},\dots, v_1, v_0)$  for all $\mathbf{v}\in\mathbb{R}^{n+1}$;
\item $\mathfrak{h}\left(v_0, v_1,\dots, v_{n-1}, v_n \right) = \mathfrak{h}(v_n, v_{n-1},\dots, v_1, v_0)$ for all $\mathbf{v}\in\mathbb{R}^{n+1}$,
\end{enumerate}
is called a \emph{height function}.
\end{definition}
Note that the last two conditions in the definition correspond to \eqref{eq:phieven} and \eqref{eq:phiinv} respectively.

Now we can define the respective height function for algebraic numbers. Let $p$ be the minimal polynomial of an algebraic number $\alpha$.
Clearly, for any given algebraic $\alpha$, the same minimal polynomial $p$ is obtained for any height function $\mathfrak{h}$, and the height $\mathfrak{h}(\alpha)$ can be defined as $\mathfrak{h}(p)$.

Let us define the distribution of algebraic numbers with respect to the height function $\mathfrak{h}(\alpha)$.
In the general case, the counting density with respect to $\mathfrak{h}$ assumes the~form
\[
\phi_n(\mathfrak{h};x) = \int\limits_{\Delta_n(\mathfrak{h};x)} \left|\sum_{k=1}^n k p_k x^{k-1}\right| dp_1\dots dp_n,
\]
where
\[
\Delta_n(\mathfrak{h};x) = \left\{(p_1,\dots,p_n)\in\mathbb{R}^n : \mathfrak{h}\!\left(p_n, \dots, p_2, p_1, -\sum_{k=1}^n p_k x^k\right) \le 1 \right\}.
\]
Naturally, this function behaves differently for different height functions $\mathfrak{h}$.

To illustrate this fact, let us calculate the counting density with respect to the spherical norm $\|\cdot\|_2$. The region $\mathcal{D}(I)$ defined by \eqref{eq-DI-def} is then replaced by two $(n+1)$-dimensional spherical wedges of radius 1 with an acute angle $\theta(\alpha,\beta)$ formed by two planes with normal vectors $\mathbf{n}_1 = (\alpha^n,\dots,\alpha,1)$ and $\mathbf{n}_2 = (\beta^n,\dots,\beta,1)$. The volume of one wedge equals
\[
\frac{\pi^{\frac{n+1}{2}}}{\Gamma\left(\frac{n+3}{2}\right)}\: \frac{\theta}{2\pi},
\]
and thus the formula \eqref{eq-limit-DI} can be written as
\[
\phi_n(\|\cdot\|_2; \alpha) = \frac{\pi^{\frac{n-1}{2}}}{\Gamma\left(\frac{n+3}{2}\right)} \lim_{\beta\to \alpha} \frac{\theta(\alpha,\beta)}{\beta-\alpha}.
\]
The angle $\theta(\alpha,\beta)$ tends to zero as $\beta$ tends to $\alpha$, allowing us to substitute $\theta$ by $\sin\theta$,
which can be calculated from the scalar product of $\mathbf{n}_1$ and $\mathbf{n}_2$:
\[
\sin \theta = \sqrt{1- \frac{(\mathbf{n}_1 \mathbf{n}_2)^2}{\mathbf{n}_1^2 \mathbf{n}_2^2}}
\]
Note that if a function $f:\mathbb{R}^2 \to \mathbb{R}$ is three times continuously differentiable and $f(x,y)=f(y,x)$ for any $x,y\in \mathbb{R}$, then by the Taylor development of $f(\cdot,\cdot)$ at $(x,x)$ as $y$ tends to $x$, one can obtain
\begin{multline*}
f(x,x)f(y,y) - (f(x,y))^2 = \\
= \left(f(x,x) \left.\frac{\partial^2 f}{\partial x\partial y}\right|_{y=x} - \left(\left.\frac{\partial f}{\partial x}\right|_{y=x}\right)^2\right) (y-x)^2 + O\left((y-x)^3\right).
\end{multline*}
Taking $f(x,y) = \sum_{k=0}^n (xy)^k = \frac{(xy)^{n+1}-1}{xy-1}$, we obtain after transformations
\begin{multline*}
\mathbf{n}_1^2 \mathbf{n}_2^2 - (\mathbf{n}_1 \mathbf{n}_2)^2 = \\
= \frac{(\alpha^{2n+2}-1)^2-(n+1)^2 \alpha^{2n} (\alpha^2-1)^2}{(\alpha^2-1)^4} (\beta-\alpha)^2 + O((\beta-\alpha)^3).
\end{multline*}
Finally, we have
\[
\phi_n(\|\cdot\|_2; x) = \frac{\pi^{\frac{n-1}{2}}}{\Gamma\left(\frac{n+3}{2}\right)} \;
\sqrt{\frac{1}{(x^2-1)^2} - \frac{(n+1)^2 x^{2n}}{(x^{2n+2}-1)^2}}.
\]
Clearly, this function is quite different from \eqref{eq-analit-f}. Up to the constant factor, the function $\phi_n(\|\cdot\|_2; x)$ coincides with the density function of zeros of random polynomials of $n$-th degree with independent identically normally distributed coefficients
(see \cite{Kac43}). In the excellent paper by A.~Edelman and E.~Kostlan \cite[\S\S 2.2--2.4]{EdeKos95}, one can find an interesting geometrical interpretation of this case.

\medskip

\paragraph{\bf Remark 4}
The main result can be reformulated in terms of sequences.

\begin{corollary}
Let the real algebraic numbers of degree $n$ are ordered in a~sequence $(\alpha_i)_{i=1}^\infty$ in such a~way that $H(\alpha_i) \le H(\alpha_{i+1})$ for any $i\ge 1$.
Let $\mathfrak{N}_n(N, S)$ be the number of elements of the truncated sequence $(\alpha_i)_{i=1}^N$ lying in a set $S$.
Then for any interval $I \subseteq \mathbb{R}$ we have
\[
\mathfrak{N}_n(N, I) = N\int\limits_I \rho_n(x)\,dx + O\!\left(N^{\frac{n}{n+1}} (\ln N)^{\ell(n)}\right),
\]
where the implicit constant in the big-O notation depends only on the degree~$n$.
\end{corollary}

\begin{proof}
Let us denote $H(\alpha_N)$ by $Q_N$.
By the definition of the sequence $(\alpha_i)_{i=1}^\infty$, we have
\begin{gather}
\#\mathbb{A}_n(Q_N-1) \le N \le \#\mathbb{A}_n(Q_N), \label{eq-N-QN}\\
\Phi_n(Q_N-1, I) \le \mathfrak{N}_n(N, I) \le \Phi_n(Q_N, I). \label{eq-NI}
\end{gather}
Theorem \ref{thm1} yields
\[
\#\mathbb{A}_n(Q) = \frac{\gamma_n Q^{n+1}}{2\zeta(n+1)} + O\!\left(Q^n (\ln Q)^{\ell(n)}\right),
\]
where
\[
\gamma_n = \int\limits_{-\infty}^{+\infty} \phi_n(x)\, dx.
\]
Thus, from \eqref{eq-N-QN}, we obtain
\[
Q_N^{n+1} = \frac{N}{\gamma_n} + O\!\left(N^{\frac{n}{n+1}} (\ln N)^{\ell(n)}\right).
\]
Applying \eqref{eq_thm1}, \eqref{eq-NI} and the identity $\rho_n(x)=\gamma_n^{-1} \phi_n(x)$ proves the corollary.
\end{proof}

It is easy to construct an example of such ordering: if $\alpha$ and $\beta$ are algebraic numbers of degree $n$, let $\alpha$ precede $\beta$ if and only if $H(\alpha) < H(\beta)$ or $H(\alpha)=H(\beta)$ and $\alpha<\beta$:
\[
\alpha \prec \beta \iff
\left[
\begin{array}{l}
H(\alpha) < H(\beta),\\
H(\alpha) = H(\beta), \quad \alpha < \beta.
\end{array}
\right.
\]

Note that the ordering imposed on algebraic numbers with identical degrees and heights can be arbitrary: the number of algebraic numbers $\alpha$ with $\deg(\alpha) = n$ and $H(\alpha) = Q$ equals $O(Q^n)$, whereas for $H(\alpha)\le Q$ it is $O(Q^{n+1})$.

\section*{Acknowledgements}
The author is grateful to Prof. Vasili Bernik and Prof. Friedrich G{\"o}tze and Dr. Dmitry Zaporozhets and Prof. Victor Beresnevich for their interesting and helpful comments on historical background of the problem and on relations between the results of this paper and studies in the area of random polynomials.
The author wishes to thank the employees of the University of Bielefeld, where a substantial part of this work was done, for providing a stimulating research environment during his visits supported by CRC 701.
The author would like to thank Dr. Nikolai Kalosha for his assistance in writing the manuscript.
Thanks are also due to the two anonymous referees for their valuable remarks.

%\bibliographystyle{siam}
%\bibliography{algnums5e}

\end{document}